\documentclass[12pt,reqno]{article}
\usepackage[usenames]{color}
\usepackage{amssymb}
\usepackage{graphicx}
\usepackage{amscd}

\usepackage{graphics,amsmath,amssymb,relsize}
\usepackage{bm}
\usepackage{amsthm}
\usepackage{amsfonts}
\usepackage{latexsym}
\usepackage{epsf}

\newtheorem{theorem}{Theorem}[section]
\newtheorem{corollary}[theorem]{Corollary}
\newtheorem{lemma}[theorem]{Lemma}

\newtheorem{defin}[theorem]{Definition}

\newtheorem{examp}[theorem]{Example}
\newenvironment{example}{\begin{examp}\normalfont\quad}{\end{examp}}
\newtheorem{rema}[theorem]{Remark}
\newtheorem{prob}[theorem]{Problem}

\numberwithin{equation}{section}

\newcommand{\bt}{\begin{thm}}
\newcommand{\et}{\end{thm}}
\newcommand{\bp}{\begin{proof}}
\newcommand{\ep}{\end{proof}}
\newcommand{\bprop}{\begin{prop}}
\newcommand{\eprop}{\end{prop}}
\newcommand{\bl}{\begin{lemma}}
\newcommand{\el}{\end{lemma}}
\newcommand{\bc}{\begin{corollary}}
\newcommand{\ec}{\end{corollary}}
\newcommand{\Z}{\mathbb{Z}}
\newcommand{\C}{\mathbb{C}}
\newcommand{\be}{\begin{enumerate}}
\newcommand{\ee}{\end{enumerate}}

\newcommand{\OMIT}[1]{}

\title{Counting surface-kernel epimorphisms \\ from a co-compact Fuchsian group to a cyclic group \\ with motivations from string theory and QFT}
\author{Khodakhast Bibak \thanks{Department of Computer Science, University of Victoria, Victoria, BC, Canada V8W 3P6. Email: {\tt \{kbibak,bmkapron,srinivas\}@uvic.ca}} \and Bruce M. Kapron \footnotemark[1] \and Venkatesh Srinivasan \footnotemark[1]}

\begin{document}

\maketitle

\begin{abstract}
Graphs embedded into surfaces have many important applications, in particular, in combinatorics, geometry, and physics. For example, ribbon graphs and their counting is of great interest in string theory and quantum field theory (QFT). Recently, Koch, Ramgoolam, and Wen [Nuclear Phys.\,B {\bf 870} (2013), 530--581] gave a refined formula for  counting ribbon graphs and discussed its applications to several physics problems. An important factor in this formula is the number of surface-kernel epimorphisms from a co-compact Fuchsian group to a cyclic group. The aim of this paper is to give an explicit and practical formula for the number of such epimorphisms. As a consequence, we obtain an `equivalent' form of the famous Harvey's theorem on the cyclic groups of automorphisms of compact Riemann surfaces. Our main tool is an explicit formula for the number of solutions of restricted linear congruence recently proved by Bibak et al. using properties of Ramanujan sums and of the finite Fourier transform of arithmetic functions.
\end{abstract}


\section{Introduction}\label{Sec_1}

A {\it surface} is a compact oriented two-dimensional topological manifold. Roughly speaking, a surface is a space that `locally' looks like the Euclidean plane. Informally, a graph is said to be {\it embedded into} (or {\it drawn on}) a surface if it can be drawn on the surface in such a way that its edges meet only at their endpoints. A {\it ribbon graph} is a finite and connected graph together with a cyclic ordering on the set of half edges incident to each vertex. One can see that ribbon graphs and embedded graphs are essentially equivalent concepts; that is, a ribbon graph can be thought as a set of disks (or vertices) attached to each other by thin stripes (or edges) glued to their boundaries. There are several other names for these graphs in the literature, for example, {\it fat graphs}, or {\it combinatorial maps}, or {\it unrooted maps}. For a thorough introduction to the theory of embedded graphs we refer the reader to the lovely book by Lando and Zvonkin \cite{LZ}.

Graphs embedded into surfaces have many important applications, in particular, in combinatorics, geometry, and physics. For example, ribbon graphs and their counting is of great interest in string theory and quantum field theory (QFT). Here we quote some of these applications and motivations from \cite{KR, KRW}:

\begin{itemize}

   \item Ribbon graphs arise in the context of MHV rules for constructing amplitudes. In the MHV rules approach to amplitudes, inspired by twistor string theory, amplitudes are constructed by gluing MHV vertices. Counting ribbon graphs play an important role here in finding different ways of gluing the vertices which contribute to a given amplitude.

   \item The number of ribbon graphs is the fundamental combinatorial element in perturbative large $N$ QFT computations, since we need to be able to enumerate the graphs and then compute corresponding Feynman integrals. 

   \item In matrix models (more specifically, the Gaussian Hermitian and complex matrix models), which can be viewed as QFTs in zero dimensions, the correlators are related very closely to the combinatorics of ribbon graphs. There is also a two-dimensional structure (related to string worldsheets) to this combinatorics.

   \item There is a bijection between vacuum graphs of Quantum Electrodynamics (QED) and ribbon graphs. In fact, the number of QED/Yukawa vacuum graphs with $2v$ vertices is equal to the number of ribbon graphs with $v$ edges. This can be proved using permutations. Note that QED is an Abelian gauge theory with the symmetry circle group $U(1)$.

\end{itemize}

Mednykh and Nedela \cite{MENE} obtained a formula for the number of unrooted maps of a given genus. Recently, Koch, Ramgoolam, and Wen \cite{KRW} gave a refinement of that formula to make it more suitable for applications to several physics problems, like the ones mentioned above. In both formulas, there is an important factor, namely, the number of surface-kernel epimorphisms from a co-compact Fuchsian group to a cyclic group. A formula for the number of such epimorphisms is given in \cite{MENE} but that formula does not seem to be very applicable, especially for large values, because one needs to find, as part of the formula, a challenging summation involving the products of some Ramanujan sums and for each index of summation one needs to calculate these products. The aim of this paper is to give a very explicit and practical formula for the number of such epimorphisms. Our formula does not contain Ramanujan sums or other challenging parts, and is really easy to work with. As a consequence, we obtain an `equivalent' form of the famous Harvey's theorem on the cyclic groups of automorphisms of compact Riemann surfaces.

In the next section, we review Fuchsian groups and Harvey's theorem. Our main tool in this paper is an explicit formula for the number of solutions of restricted linear congruence recently proved by Bibak et al. \cite{BKSTT} using properties of Ramanujan sums and of the finite Fourier transform of arithmetic functions, which is reviewed in Section~\ref{Sec_3}. Our main result is presented in Section~\ref{Sec_4}.

\section{Fuchsian groups and Harvey's theorem}\label{Sec_2}

A {\it Fuchsian group} $\Gamma$ is a finitely generated non-elementary discrete subgroup of $\text{PSL}(2,\mathbb{R})$, the group of orientation-preserving isometries of the hyperbolic plane $\mathbb{H}^2$. Fuchsian groups were first studied by Poincar\'{e} in 1882 in connection with the uniformization problem (later the uniformization theorem), and he called the groups Fuchsian after Lazarus Fuchs whose paper (1880) was a motivation for Poincar\'{e} to introduce this concept. By a classical result of Fricke and Klein (see, e.g., \cite{ZVC}), every such group $\Gamma$ has a presentation in terms of the generators $\textsf{a}_1,\textsf{b}_1,\ldots,\textsf{a}_g,\textsf{b}_g$ (hyperbolic), $\textsf{x}_1,\ldots,\textsf{x}_k$ (elliptic), $\textsf{y}_1,\ldots,\textsf{y}_s$ (parabolic), and $\textsf{z}_1,\ldots,\textsf{z}_t$ (hyperbolic boundary elements) with the relations
\begin{align} \label{Fuchsian relations}
\textsf{x}_1^{n_1} = \cdots = \textsf{x}_k^{n_k} = \textsf{x}_1 \cdots \textsf{x}_k\textsf{y}_1 \cdots \textsf{y}_s\textsf{z}_1 \cdots \textsf{z}_t[\textsf{a}_1, \textsf{b}_1]\cdots [\textsf{a}_g, \textsf{b}_g]=1,
\end{align}
where $k,s,t,g \geq 0$, $n_i\geq 2$ $(1\leq i \leq k)$, and $[a, b] = a^{-1}b^{-1}ab$. The integers $n_1,\ldots,n_k$ are called the {\it periods} of $\Gamma$, and $g$ is called the {\it orbit genus}. The Fuchsian group $\Gamma$ is determined, up to isomorphism, by the tuple $(g;n_1,\ldots,n_k;s;t)$ which is referred to as the {\it signature} of $\Gamma$. If $k=0$ (i.e., there are no periods), $\Gamma$ is called a Fuchsian {\it surface group}. If $s=t=0$, the group is called {\it co-compact} (or {\it F-group}, or {\it proper}). Some authors by a Fuchsian group mean a co-compact Fuchsian group. In this paper, we only work with co-compact Fuchsian groups.

We denotes by $\text{Hom}(\Gamma,G)$ (resp., $\text{Epi}(\Gamma,G)$) the set of homomorphisms (resp., epimorphisms) from a Fuchsian group $\Gamma$ to a finite group $G$. There is much interest (with many applications) in combinatorics, geometry, algebra, and physics, in counting homomorphisms and epimorphisms from a Fuchsian group to a finite group. For example, Liebeck and Shalev \cite{LS1, LS2} obtained good estimates for $|\text{Hom}(\Gamma,G)|$, where $\Gamma$ is an arbitrary Fuchsian group and $G$ is a symmetric group or an alternating group or a finite simple group.

An epimorphism from a Fuchsian group to a finite group with kernel a Fuchsian surface group is called {\it surface-kernel} (or {\it smooth}). Harvey proved that an epimorphism $\phi$ from a co-compact Fuchsian group $\Gamma$ to a finite group $G$ is surface-kernel if and only if it preserves the periods of $\Gamma$, that is, for every elliptic generator $\textsf{x}_i$ $(1\leq i \leq k)$ of order $n_i$, the order of $\phi(\textsf{x}_i)$ is precisely $n_i$. The above-mentioned equivalence appears in Harvey's influential 1966 paper \cite{HAR} on the cyclic groups of automorphisms of compact Riemann surfaces. The main result of this paper is the following theorem which gives necessary and sufficient conditions for the existence of a surface-kernel epimorphism from a co-compact Fuchsian group to a cyclic group.

\begin{theorem} {\rm (\cite{HAR})} \label{Harvey Thm}
Let $\Gamma$ be a co-compact Fuchsian group with signature $(g;n_1,\ldots,n_k)$, and let $\mathfrak{n}:=\textnormal{lcm}(n_1,\ldots ,n_k)$. There is a surface-kernel epimorphism from $\Gamma$ to $\Z_n$ 
if and only if the following conditions are satisfied:

(i) $\textnormal{lcm}(n_1,\ldots ,n_{i-1}, n_{i+1},\ldots, n_k)=\mathfrak{n}$, for all $i$;

(ii) $\mathfrak{n}\mid n$, and if $g = 0$ then $\mathfrak{n}=n$;

(iii) $k\not=1$, and, if $g=0$ then $k>2$;

(iv) if $\mathfrak{n}$ is even then the number of periods $n_i$ such that $\mathfrak{n}/n_i$ is odd is also even.
\end{theorem}

By a result of Burnside \cite{BUR}, and of Greenberg \cite{GRE}, every finite group $G$ acts as a group of automorphisms of a compact Riemann surface of genus at least two. The {\it minimum genus} problem asks to find, for a given finite group $G$, the minimum genus of those compact Riemann surfaces on which $G$ acts faithfully as a group of conformal automorphisms. Harvey \cite{HAR}, using Theorem~\ref{Harvey Thm}, solved the minimum genus problem when $G$ is the cyclic group $\Z_n$; in fact, he gave an explicit value for the minimum genus in terms of the prime factorization of $n$. Then, as a corollary, he obtained a famous result of Wiman \cite{WIM} on the {\it maximum order} for an automorphism of a compact Riemann surface of genus $\gamma$ by showing that this maximum order is $2(2\gamma+1)$.

Harvey's paper \cite{HAR} played a pioneering role in studying groups of automorphisms of compact Riemann surfaces and also has found important applications in some other areas of mathematics like combinatorics. See, for example, the survey by Bujalance et al. \cite{BCG} on the ``research inspired by Harvey's theorem", in which the authors describe many results about the actions of several classes of groups, including cyclic, Abelian, solvable, dihedral, etc., along with the minimum genus and maximum order problems for these classes.

\section{Ramanujan sums and restricted linear congruences}\label{Sec_3}

Throughout the paper we use $\gcd(a_1,\ldots,a_k)$ and $\textnormal{lcm}(a_1,\ldots,a_k)$ to denote, respectively, the greatest common divisor and the least common multiple of integers $a_1,\ldots,a_k$. For $a \in \Z \setminus \lbrace 0 \rbrace$ and $b \in \Z$, by $a \mid b$ and $a \nmid b$ we mean, respectively, $a$ is a divisor of $b$, and $a$ is not a divisor of $b$. Also, for $a \in \Z \setminus \lbrace 0 \rbrace$ and a prime $p$ we use the notation $p^r\, \|\, a$ if $p^r\mid a$ and $p^{r+1}\nmid a$. A function $f:\Z \to \C$ is called {\it periodic} with period $n$ (also called {\it $n$-periodic} or {\it periodic} modulo $n$) if $f(m + n) = f(m)$, for every $m\in \mathbb{Z}$. In this case $f$ is determined by the finite vector $(f(1),\ldots,f(n))$.

Let $e(x)=\exp(2\pi ix)$ be the complex exponential with period 1. For integers $m$ and $n$ with $n \geq 1$, the quantity 
\begin{align}\label{def1}
c_n(m) = \mathlarger{\sum}_{\substack{j=1 \\ (j,n)=1}}^{n}
e\!\left(\frac{jm}{n}\right)
\end{align}
is called a {\it Ramanujan sum}, which is also denoted by $c(m,n)$ in the literature. From (\ref{def1}) it is clear that $c_n(m)$ is a periodic function of $m$ with period $n$.

Clearly, $c_n(0)=\varphi (n)$, where $\varphi (n)$ is {\it Euler's totient function}. Also, $c_n(1)=\mu (n)$, where $\mu (n)$ is the {\it M\"{o}bius function} defined by
\begin{align}\label{def2}
 \mu (n)&=
  \begin{cases}
    1, & \text{if $n=1$,}\\
    0, & \text{if $n$ is not square-free,}\\
    (-1)^{\kappa}, & \text{if $n$ is the product of $\kappa$ distinct primes}.
  \end{cases}
\end{align}

The classical version of the M\"{o}bius inversion formula states that if $f$ and $g$ are arithmetic functions satisfying $g(n)=\sum_{d\, \mid\, n} f(d)$, for every integer $n\geq 1$, then 
\begin{align} \label{inver form}
f(n)=\mathlarger{\sum}_{d\, \mid\, n} \mu \!\left(\frac{n}{d}\right)g(d),
\end{align}
for every integer $n\geq 1$.

Let $a_1,\ldots,a_k,b,n\in \Z$, $n\geq 1$. A linear congruence in $k$ unknowns $x_1,\ldots,x_k$ is of the form
\begin{align} \label{cong form}
a_1x_1+\cdots +a_kx_k\equiv b \pmod{n}.
\end{align}
By a solution of (\ref{cong form}), we mean an ordered $k$-tuple $\mathbf{x}=\langle x_1,\ldots,x_k \rangle \in \mathbb{Z}_n^k$ that satisfies (\ref{cong form}).

The solutions of the above congruence may be subject to certain conditions, such as $\gcd(x_i,n)=t_i$ ($1\leq i\leq k$), where $t_1,\ldots,t_k$ are given positive divisors of $n$. The number of solutions of these congruences, which were called {\it restricted linear congruences} in \cite{BKSTT}, was first considered by Rademacher \cite{Rad1925} in 1925 and Brauer \cite{Bra1926} in 1926, in the special case of $a_i=t_i=1$ $(1\leq i \leq k)$ (see Corollary~\ref{RB thm} below). Since then, this problem has been studied, in several other special cases, in many papers (for example, Cohen \cite{COH2} dealt with the special case of $t_i=1$, $a_i \mid n$, $a_i$ prime, for all $i$) and has found very interesting applications in number theory, combinatorics, computer science, and cryptography, among other areas; see \cite{BKSTT3, BKSTT, BKSTT2} for a detailed discussion about this problem and a comprehensive list of references. Recently, Bibak et al. \cite{BKSTT} dealt with the problem in its `most general case' and using properties of Ramanujan sums and of the finite Fourier transform of arithmetic functions gave an explicit formula for the number of solutions of the restricted linear congruence
\begin{equation} \label{gen_rest_cong}
a_1x_1+\cdots +a_kx_k\equiv b \pmod{n}, \quad \gcd(x_i,n)=t_i \ (1\leq i\leq k),
\end{equation}
where $a_1,t_1,\ldots,a_k,t_k, b,n$ ($n\geq 1$) are arbitrary integers.

\begin{theorem} {\rm (\cite{BKSTT})} \label{thm:k var1 more} Let $a_i,t_i, b,n\in \Z$, $n\geq 1$, $t_i\mid n$ {\rm ($1\leq i\leq k$)}. The number of solutions of the linear congruence $a_1x_1+\cdots +a_kx_k\equiv b \pmod{n}$, with $\gcd(x_i,n)=t_i$ {\rm ($1\leq i\leq k$)}, is
\begin{align} \label{thm:k var1 more for}
N_n(b;a_1,t_1,\ldots,a_k,t_k) & =\frac{1}{n}\left(\mathlarger{\prod}_{i=1}^{k}\frac{\varphi\left(\frac{n}{t_i}\right)}{\varphi\left(\frac{n}{t_id_i}\right)}\right)
\mathlarger{\sum}_{d\, \mid\, n}
c_d(b)\mathlarger{\prod}_{i=1}^{k}c_{\frac{n}{t_id_i}}\! \left(\frac{n}{d}\right)\\
& = \label{thm:k var2 more for}
\frac{1}{n} \left(\mathlarger{\prod}_{i=1}^{k}  \varphi\left(\frac{n}{t_i}\right)\right) \mathlarger{\sum}_{d\, \mid\, n}
c_d(b) \mathlarger{\prod}_{i=1}^{k} \frac{\mu\left(\frac{d}{\gcd(a_it_i,d)}\right)}{\varphi\left(\frac{d}{\gcd(a_it_i,d)}\right)},
\end{align}
where $d_i=\gcd(a_i,\frac{n}{t_i})$ {\rm ($1\leq i\leq k$)}.
\end{theorem}

While Theorem~\ref{thm:k var1 more} is useful from several aspects (for example, we use it in the proof of Theorem~\ref{sepi result}), for many applications (for example, the ones considered in this paper) we need a more explicit formula.

If in \eqref{gen_rest_cong} one has $a_i=0$ for every $1\leq i\leq k$, then clearly there are solutions 
$\langle x_1,\ldots,x_k\rangle$ if and only if $b\equiv 0\pmod{n}$ and $t_i \mid n$ ($1\leq i\leq k$), and in this case there are $\varphi(n/t_1)\cdots \varphi(n/t_k)$ solutions.

Consider the restricted linear congruence \eqref{gen_rest_cong} and assume that there is an $i_0$ such that $a_{i_0}\ne 0$. For every prime divisor $p$ of $n$ let $r_p$ be the exponent of $p$ in the prime factorization of $n$ and let $\mathfrak{m}_p=\mathfrak{m}_p(a_1,t_1,\ldots,a_k,t_k)$ denote the smallest $j\geq 1$ such that there is some $i$ with $p^j \nmid a_it_i$. There exists a finite $\mathfrak{m}_p$ for every $p$, since for a sufficiently large $j$ one has $p^j\nmid a_{i_0}t_{i_0}$. Furthermore, let
$$
e_p = e_p(a_1,t_1,\ldots,a_k,t_k) = \# \{i: 1\leq i\leq k, p^{\mathfrak{m}_p}\nmid a_it_i \}.
$$
By definition, $1 \leq e_p \leq$ the number of $i$ such that $a_i\ne 0$. Note that in many situations instead of $\mathfrak{m}_p(a_1,t_1,\ldots,a_k,t_k)$ we write $\mathfrak{m}_p$ and instead of $e_p(a_1,t_1,\ldots,a_k,t_k)$ we write $e_p$ for short. However, it is important to note that both $\mathfrak{m}_p$ 
and $e_p$ always depend on $a_1,t_1,\ldots,a_k,t_k,p$.

\begin{theorem} {\rm (\cite{BKSTT})} \label{th_gen_expl} Let $a_i,t_i, b,n\in \Z$, $n\geq 1$, $t_i\mid n$ {\rm ($1\leq i\leq k$)} and assume that $a_i\neq 0$ for at least one $i$. Consider the linear congruence $a_1x_1+\cdots +a_kx_k\equiv b \pmod{n}$, with $\gcd(x_i,n)=t_i$ {\rm ($1\leq i\leq k$)}. If there is a prime $p\mid n$ such that $\mathfrak{m}_p\leq r_p$ and $p^{\mathfrak{m}_p-1}\nmid b$ or $\mathfrak{m}_p\geq r_p+1$ and $p^{r_p}\nmid b$, then the linear congruence has no solution. Otherwise, the number of solutions is
\begin{equation} \label{main_prod_formula}
\mathlarger{\prod}_{i=1}^{k} \varphi\left(\frac{n}{t_i}\right)
\mathlarger{\prod}_{\substack{p\,\mid\, n \\ \mathfrak{m}_p \,\leq \, r_p \\ p^{\mathfrak{m}_p} \,\mid\, b}} p^{\mathfrak{m}_p-r_p-1} \left(1-\frac{(-1)^{e_p-1}}{(p-1)^{e_p-1}} \right)
\mathlarger{\prod}_{\substack{p\, \mid\, n \\ \mathfrak{m}_p \,\leq \, r_p \\ p^{\mathfrak{m}_p-1} \, \|\, b}} p^{\mathfrak{m}_p-r_p-1} \left(1-\frac{(-1)^{e_p}}{(p-1)^{e_p}}\right),
\end{equation}
where the last two products are over the prime factors $p$ of $n$ with the given additional properties. Note that the last product is empty and equal to $1$ if $b=0$.
\end{theorem}

Interestingly, if in Theorem~\ref{th_gen_expl} we put $a_i=t_i=1$ $(1\leq i \leq k)$ then we get the following result first proved by Rademacher \cite{Rad1925} in 1925 and Brauer \cite{Bra1926} in 1926.

\begin{corollary} \label{RB thm}
Let $b,n\in \Z$ and $n\geq 1$. The number of solutions of the linear congruence $x_1+\cdots +x_k\equiv b \pmod{n}$, with $\gcd(x_i,n)=1$ ($1\leq i\leq k$) is
\begin{align} \label{rest_cong_1}
\frac{\varphi(n)^k}{n} \mathlarger{\prod}_{p\, \mid \, n, \,
p\, \mid\, b} \!\left(1-\frac{(-1)^{k-1}}{(p-1)^{k-1}}
\right)\mathlarger{\prod}_{p\, \mid\, n, \, p \, \nmid \, b}
\!\left(1-\frac{(-1)^k}{(p-1)^k}\right).
\end{align}
\end{corollary}

\begin{proof}
Since $a_i=t_i=1$ $(1\leq i \leq k)$, for every prime divisor $p$ of $n$ we have $\mathfrak{m}_p=1$ and 
$e_p=k$. So, for every prime divisor $p$ of $n$ we also have $\mathfrak{m}_p=1 \leq r_p$. Clearly, the first part of Theorem~\ref{th_gen_expl} does not hold in this special case, that is, there is no prime $p\mid n$ such that $\mathfrak{m}_p\leq r_p$ and $p^{\mathfrak{m}_p-1}\nmid b$ or $\mathfrak{m}_p\geq r_p+1$ and $p^{r_p}\nmid b$. Furthermore, we have
\begin{align*}
\mathlarger{\prod}_{p\, \mid \, n, \,
p\, \mid\, b} \!p^{r_p}\mathlarger{\prod}_{p\, \mid\, n, \, p \, \nmid \, b}
\!p^{r_p}=n.
\end{align*}
Thus, the result follows by a simple application of the second part of Theorem~\ref{th_gen_expl}, \eqref{main_prod_formula}.
\end{proof}

We note that, while Theorem~\ref{th_gen_expl} may seem a bit complicated,  it is in fact easy to work with; see \cite[Ex. 3.11]{BKSTT} where we show, via several examples, how to apply Theorem~\ref{th_gen_expl}.

Theorem~\ref{th_gen_expl} is the only result in the literature which gives \textit{necessary and sufficient conditions} for the (non-)existence of solutions of restricted linear congruences in their most general case (see Corollary~\ref{cor_zero} below) and might lead to interesting applications/implications. For example, see \cite{BKSTT2} for applications in computer science and cryptography, and \cite{BKSTT} for connections to the generalized knapsack problem proposed by Micciancio. In this paper, we give more applications for this result.

\begin{corollary} {\rm (\cite{BKSTT})} \label{cor_zero} The restricted congruence given in Theorem~\ref{th_gen_expl} has no solutions if and only if one of the following cases holds:

(i) there is a prime $p\mid n$ with $\mathfrak{m}_p\leq r_p$ and $p^{\mathfrak{m}_p-1}\nmid b$;

(ii) there is a prime $p\mid n$ with $\mathfrak{m}_p\geq r_p+1$ and $p^{r_p}\nmid b$;

(iii) there is a prime $p\mid n$ with $\mathfrak{m}_p\leq r_p$, $e_p=1$ and $p^{\mathfrak{m}_p}\mid b$;

(iv) $n$ is even, $\mathfrak{m}_2 \leq r_2$, $e_2$ is odd and $2^{\mathfrak{m}_2}\mid b$;

(v) $n$ is even, $\mathfrak{m}_2 \leq r_2$, $e_2$ is even and $2^{\mathfrak{m}_2-1}\, \|\, b$.
\end{corollary}

\section{Counting surface-kernel epimorphisms from $\Gamma$ to $\Z_n$}\label{Sec_4}

In this section, we obtain an explicit formula for the number of surface-kernel epimorphisms from a co-compact Fuchsian group to a cyclic group. First, we need a formula that connects the number of epimorphisms to the number of homomorphisms as, generally, enumerating homomorphisms is easier than enumerating epimorphisms.

The M\"{o}bius function and M\"{o}bius inversion were studied for functions over locally finite partially ordered sets (posets) first by Weisner \cite{WEI} and Hall \cite{HAL}, motivated by group theory problems. Later, Rota \cite{ROT} extended this idea and put it in the context of combinatorics. Following the argument given in \cite{HAL}, we prove the following simple result.

\begin{theorem} \label{epi-hom}
Let $\Lambda$ be a finitely generated group. Then

\begin{equation} \label{epi-hom_formula}
|\textnormal{Epi}(\Lambda,\Z_n)| = \mathlarger{\sum}_{d\, \mid\, n} \mu \!\left(\frac{n}{d}\right)|\textnormal{Hom}(\Lambda,\Z_d)|,
\end{equation}
where the summation is taken over all positive divisors $d$ of $n$.
\end{theorem}

\begin{proof}
Clearly, for a finitely generated group $\Lambda$ and a finite group $G$ we have 

\begin{equation*}
|\textnormal{Hom}(\Lambda,G)| = \mathlarger{\sum}_{H\leq G} |\textnormal{Epi}(\Lambda,H)|,
\end{equation*}
because every homomorphism from $\Lambda$ to $G$ induces a unique epimorphism from $\Lambda$ to its image in $G$.

Taking $G=\Z_n$, and since the cyclic group $\Z_n$ has a unique subgroup $\Z_d$ for every positive divisor $d$ of $n$ and has no other subgroups, we get

\begin{equation*}
|\textnormal{Hom}(\Lambda,\Z_n)| = \mathlarger{\sum}_{d\, \mid\, n} |\textnormal{Epi}(\Lambda,\Z_d)|.
\end{equation*}
Now, by applying the M\"{o}bius inversion formula, \eqref{inver form}, the theorem follows.
\end{proof}

We also need the following well-known result which gives equivalent defining formulas for {\it Jordan's totient function} $J_{k}(n)$ (see, e.g., \cite[pp. 13-14]{MCC}).

\begin{lemma} \label{mu pro}
Let $n$, $k$ be positive integers. Then 
\begin{equation} \label{mu pro formula}
J_{k}(n) = \mathlarger{\sum}_{d\, \mid\, n}d^k\mu \!\left(\frac{n}{d}\right) = n^k\mathlarger{\prod}_{p\, \mid\, n}\left(1-\frac{1}{p^k}\right),
\end{equation}
where the left summation is taken over all positive divisors $d$ of $n$, and the right product is taken over all prime divisors $p$ of $n$.
\end{lemma}

Now, using the above results, we obtain an explicit formula for the number $|\text{Epi}_{\text{S}}(\Gamma,\Z_n)|$ of surface-kernel epimorphisms from a co-compact Fuchsian group $\Gamma$ to the cyclic group 
$\Z_n$.

\begin{theorem} \label{sepi result}
Let $\Gamma$ be a co-compact Fuchsian group with signature $(g;n_1,\ldots,n_k)$, and let $\mathfrak{n}:=\textnormal{lcm}(n_1,\ldots ,n_k)$. If $\mathfrak{n} \nmid n$ then there is no surface-kernel epimorphism from $\Gamma$ to $\Z_n$. Otherwise, the number of surface-kernel epimorphisms from $\Gamma$ to $\Z_n$ is
\begin{equation} \label{sepi_formula}
|\textnormal{Epi}_{\textnormal{S}}(\Gamma,\Z_n)| = \frac{n^{2g}}{\mathfrak{n}}\mathlarger{\prod}_{i=1}^{k} \varphi\left(n_i\right)\mathlarger{\prod}_{p\,\mid\, \frac{n}{\mathfrak{n}}}\left(1-\frac{1}{p^{2g}}\right)\mathlarger{\prod}_{p\,\mid\, \mathfrak{n}}\left(1-\frac{(-1)^{e_p-1}}{(p-1)^{e_p-1}} \right),
\end{equation}
where $e_p = \# \{i: 1\leq i\leq k, p\nmid \mathfrak{n}/n_i\}$.
\end{theorem}

\begin{proof}
By Theorem~\ref{epi-hom}, we have 
\begin{equation} \label{sepi-shom}
|\text{Epi}_{\text{S}}(\Gamma,\Z_n)| = \mathlarger{\sum}_{d\, \mid\, n} \mu \!\left(\frac{n}{d}\right)|\text{Hom}_{\text{S}}(\Gamma,\Z_d)|,
\end{equation}
where $|\text{Hom}_{\text{S}}(\Gamma,\Z_d)|$ is the number of surface-kernel homomorphisms from $\Gamma$ to $\Z_d$. It is easy to see that for every positive divisor $d$ of $n$ we have $|\text{Hom}_{\text{S}}(\Gamma,\Z_d)| = d^{2g}N_d$, where $N_d$ is the number of solutions of the restricted linear congruence $x_1+\cdots +x_k\equiv 0 \pmod{d}$, with $\gcd(x_i,d)=\frac{d}{n_i}$ {\rm ($1\leq i\leq k$)}. Suppose that $\mathfrak{D}:=\{d>0: d \mid n \; \mbox{and} \: \mathfrak{n}\mid d\}$. Clearly, if $\mathfrak{D}$ is empty then $|\textnormal{Hom}_{\textnormal{S}}(\Gamma,\Z_d)| = 0$, for every divisor $d$ of $n$, which then implies that $|\textnormal{Epi}_{\textnormal{S}}(\Gamma,\Z_n)|=0$, by \eqref{sepi-shom}. Let $\mathfrak{n} \nmid n$. Then $\mathfrak{n} \nmid d$, for every divisor $d$ of $n$. Thus, $\mathfrak{D}$ is empty which then implies that $|\textnormal{Epi}_{\textnormal{S}}(\Gamma,\Z_n)|=0$, by \eqref{sepi-shom}. Now, let $\mathfrak{n} \mid n$. Then there exists at least one divisor $d$ of $n$ such that $\mathfrak{n} \mid d$. So, $\mathfrak{D}$ is non-empty. Now, for every $d \in \mathfrak{D}$, by Theorem~\ref{th_gen_expl}, we have
\begin{equation}\label{nd prod}
N_d=\mathlarger{\prod}_{i=1}^{k} \varphi\left(n_i\right)
\mathlarger{\prod}_{\substack{p\,\mid\, d \\ \mathfrak{m}_p \,\leq \, r_p}} p^{\mathfrak{m}_p-r_p-1} \left(1-\frac{(-1)^{e_p-1}}{(p-1)^{e_p-1}} \right),
\end{equation}
where $r_p$ is the exponent of $p$ in the prime factorization of $d$, $\mathfrak{m}_p$ is the smallest 
$j\geq 1$ such that there is some $i$ with $p^j \nmid \frac{d}{n_i}$, and 
$
e_p = \# \{i: 1\leq i\leq k, p^{\mathfrak{m}_p}\nmid d/n_i\}.
$
On the other hand, by Theorem~\ref{thm:k var1 more}, we have 
\begin{equation*}
N_d=\frac{1}{d}\mathlarger{\sum}_{d'\, \mid\, d} \varphi(d')\mathlarger{\prod}_{i=1}^{k}c_{n_i}\! \left(\frac{d}{d'}\right),
\end{equation*}
which, as was proved in \cite[Prop.\ 9]{TOT}, equals 
\begin{equation*}
N_d=\frac{1}{d}\mathlarger{\sum}_{q=1}^{d} \mathlarger{\prod}_{i=1}^{k}c_{n_i}(q).
\end{equation*}
Now, since the Ramanujan sum $c_n(m)$ is a periodic function of $m$ with period $n$, it is easy to see (from the above equivalent expressions) that the value of $N_d$ will remain the same if we replace $d$ with $\mathfrak{n}$ in \eqref{nd prod}. Consequently, we obtain the following explicit formula for the number of surface-kernel homomorphisms from $\Gamma$ to $\Z_d$,
\begin{equation*}
|\text{Hom}_{\text{S}}(\Gamma,\Z_d)| = d^{2g}\mathlarger{\prod}_{i=1}^{k} \varphi\left(n_i\right)\mathlarger{\prod}_{\substack{p\,\mid\, \mathfrak{n} \\ \mathfrak{m}_p \,\leq \, r_p}} p^{\mathfrak{m}_p-r_p-1} \left(1-\frac{(-1)^{e_p-1}}{(p-1)^{e_p-1}} \right),
\end{equation*}
where $r_p$ is the exponent of $p$ in the prime factorization of $\mathfrak{n}$, $\mathfrak{m}_p$ is the smallest $j\geq 1$ such that there is some $i$ with $p^j \nmid \frac{\mathfrak{n}}{n_i}$, and 
$
e_p = \# \{i: 1\leq i\leq k, p^{\mathfrak{m}_p}\nmid \mathfrak{n}/n_i\}.
$

Note that since $\mathfrak{n}=\textnormal{lcm}(n_1,\ldots ,n_k)$, for every prime divisor $p$ of $\mathfrak{n}$ we have $p \nmid \frac{\mathfrak{n}}{n_i}$ for at least one $i$. This means that $\mathfrak{m}_p=1$ for every prime divisor $p$ of $\mathfrak{n}$. Also, note that
$$
\mathlarger{\prod}_{p\,\mid\, \mathfrak{n}} p^{r_p}=\mathfrak{n}.
$$
Therefore, we get 
\begin{equation*}
|\text{Hom}_{\text{S}}(\Gamma,\Z_d)| = \frac{d^{2g}}{\mathfrak{n}}\mathlarger{\prod}_{i=1}^{k} \varphi\left(n_i\right)\mathlarger{\prod}_{p\,\mid\, \mathfrak{n}} \left(1-\frac{(-1)^{e_p-1}}{(p-1)^{e_p-1}} \right),
\end{equation*}
where $e_p = \# \{i: 1\leq i\leq k, p\nmid \mathfrak{n}/n_i\}$.

Now, using (\ref{sepi-shom}), letting $d=v\mathfrak{n}$, and then using Lemma~\ref{mu pro}, we obtain 
\begin{align*}
|\textnormal{Epi}_{\textnormal{S}}(\Gamma,\Z_n)|&= \mathlarger{\sum}_{\mathfrak{n}\, \mid\, d\, \mid\, n} \mu \!\left(\frac{n}{d}\right)\frac{d^{2g}}{\mathfrak{n}}\mathlarger{\prod}_{i=1}^{k} \varphi\left(n_i\right)\mathlarger{\prod}_{p\,\mid\, \mathfrak{n}} \left(1-\frac{(-1)^{e_p-1}}{(p-1)^{e_p-1}} \right)\\
&= \mathlarger{\sum}_{v\, \mid\, \frac{n}{\mathfrak{n}}} \mu \!\left(\frac{n/\mathfrak{n}}{v}\right)v^{2g}\mathfrak{n}^{2g-1}\mathlarger{\prod}_{i=1}^{k} \varphi\left(n_i\right)\mathlarger{\prod}_{p\,\mid\, \mathfrak{n}} \left(1-\frac{(-1)^{e_p-1}}{(p-1)^{e_p-1}} \right)\\
&= \frac{n^{2g}}{\mathfrak{n}}\mathlarger{\prod}_{i=1}^{k} \varphi\left(n_i\right)\mathlarger{\prod}_{p\,\mid\, \frac{n}{\mathfrak{n}}}\left(1-\frac{1}{p^{2g}}\right)\mathlarger{\prod}_{p\,\mid\, \mathfrak{n}} \left(1-\frac{(-1)^{e_p-1}}{(p-1)^{e_p-1}} \right),
\end{align*}
where $e_p = \# \{i: 1\leq i\leq k, p\nmid \mathfrak{n}/n_i\}$.
\end{proof}

\begin{example} \label{examp epi}

1) Let $\Gamma$ be the co-compact Fuchsian group with signature $(1;2,3,4)$. Find the number of surface-kernel epimorphisms from $\Gamma$ to $\Z_{24}$. 

\bigskip

Here $\mathfrak{n}=\textnormal{lcm}(2,3,4)=12=2^2\cdot 3$. Also, $2\mid \frac{\mathfrak{n}}{n_1}=6$, $2\mid \frac{\mathfrak{n}}{n_2}=4$, $2\nmid \frac{\mathfrak{n}}{n_3}=3$. So, $e_2=1$. Therefore, by Theorem~\ref{sepi result}, we have 
$$
|\textnormal{Epi}_{\textnormal{S}}(\Gamma,\Z_{24})|=0,
$$
because 
$$
1-\frac{(-1)^{e_2-1}}{(2-1)^{e_2-1}}=1-\frac{(-1)^{1-1}}{1^{1-1}}=0.
$$
Of course, this example can be also followed directly by Harvey's theorem (Theorem~\ref{Harvey Thm}).

2) Let $\Gamma$ be the co-compact Fuchsian group with signature $(2;36,500,125,9)$. Find the number of surface-kernel epimorphisms from $\Gamma$ to $\Z_{9000}$. 

\bigskip

Here $\mathfrak{n}=\textnormal{lcm}(36,500,125,9)=\textnormal{lcm}(2^2\cdot 3^2,2^2\cdot 5^3,5^3,3^2)=2^2\cdot 3^2\cdot 5^3=4500$. We have

\bigskip

$2\nmid \frac{\mathfrak{n}}{n_1}=5^3$, $2\nmid \frac{\mathfrak{n}}{n_2}=3^2$, $2\mid \frac{\mathfrak{n}}{n_3}=2^2\cdot 3^2$, $2\mid \frac{\mathfrak{n}}{n_4}=2^2\cdot 5^3$, so, $e_2=2$; 

$3\nmid \frac{\mathfrak{n}}{n_1}=5^3$, $3\mid \frac{\mathfrak{n}}{n_2}=3^2$, $3\mid \frac{\mathfrak{n}}{n_3}=2^2\cdot 3^2$, $3\nmid \frac{\mathfrak{n}}{n_4}=2^2\cdot 5^3$, so, $e_3=2$;

$5\mid \frac{\mathfrak{n}}{n_1}=5^3$, $5\nmid \frac{\mathfrak{n}}{n_2}=3^2$, $5\nmid \frac{\mathfrak{n}}{n_3}=2^2\cdot 3^2$, $5\mid \frac{\mathfrak{n}}{n_4}=2^2\cdot 5^3$, so, $e_5=2$.

\bigskip
Now,
$$
\mathlarger{\prod}_{p\,\mid\, 4500}\left(1-\frac{(-1)^{e_p-1}}{(p-1)^{e_p-1}} \right)= \left(1-\frac{(-1)^{2-1}}{(2-1)^{2-1}}\right)\left(1-\frac{(-1)^{2-1}}{(3-1)^{2-1}} \right)\left(1-\frac{(-1)^{2-1}}{(5-1)^{2-1}} \right)= \frac{15}{4}.
$$

Therefore, by Theorem~\ref{sepi result}, we have 
$$
|\textnormal{Epi}_{\textnormal{S}}(\Gamma,\Z_{9000})|= \frac{9000^{4}}{4500}\varphi\left(2^2\cdot 3^2\right)\varphi\left(2^2\cdot 5^3\right)\varphi\left(5^3\right)\varphi\left(3^2\right)\left(1-\frac{1}{2^4}\right) \times \frac{15}{4}= 7381125\cdot 10^{12}.
$$

3) Let $\Gamma$ be the co-compact Fuchsian group with signature $(0;36,500,125,9)$. Find the number of surface-kernel epimorphisms from $\Gamma$ to $\Z_{9000}$. 

\bigskip

Here since $g=0$, 
$$
\mathlarger{\prod}_{p\,\mid\, \frac{n}{\mathfrak{n}}}\left(1-\frac{1}{p^{2g}}\right)=\mathlarger{\prod}_{p\,\mid\, 2}\left(1-\frac{1}{p^0}\right)=0.
$$
Therefore, by Theorem~\ref{sepi result}, we have 
$$
|\textnormal{Epi}_{\textnormal{S}}(\Gamma,\Z_{9000})|=0.
$$
Of course, this example can be also followed directly by Harvey's theorem.
\end{example}

\begin{rema}
In the proof of Theorem~\ref{sepi result} we have used only a special case of Theorem~\ref{th_gen_expl} where $a_i=1$ $(1\leq i \leq k)$ and $b=0$. But there may be other generalizations/variants of these or other groups so that for counting the number of surface-kernel epimorphisms (or other relevant problems) we have to use the `full power' of Theorem~\ref{th_gen_expl}.
\end{rema}

\begin{rema}
In order to get explicit values for $|\textnormal{Epi}_{\textnormal{S}}(\Gamma,\Z_n)|$ from Theorem~\ref{sepi result}, we only need to find the prime factorization of $n$, of $\mathfrak{n}$, and of the periods $n_1,\ldots,n_k$. Then we can easily compute $e_p$, $\varphi\left(n_i\right)$, etc. In fact, even for Harvey's theorem (Theorem~\ref{Harvey Thm}) we need to find these prime factorizations! So, Theorem~\ref{sepi result} has roughly the same computational cost as Harvey's theorem.
\end{rema}

Clearly, for a co-compact Fuchsian group with all periods equal to each other we have $e_p = \# \{i: 1\leq i\leq k, p\nmid \mathfrak{n}/n_i\}=k$, for every prime divisor $p$ of $\mathfrak{n}$. Therefore, we get the following simpler formula from Theorem~\ref{sepi result}.

\begin{corollary} \label{sepi eqper}
Let $\Gamma$ be a co-compact Fuchsian group with signature $(g;n_1,\ldots,n_k)$, where $n_1=\cdots=n_k=\mathfrak{n}$. If $\mathfrak{n} \nmid n$ then there is no surface-kernel epimorphism from $\Gamma$ to 
$\Z_n$. Otherwise, the number of surface-kernel epimorphisms from $\Gamma$ to $\Z_n$ is
\begin{equation} \label{sepi eqper_formula}
|\textnormal{Epi}_{\textnormal{S}}(\Gamma,\Z_n)| = \frac{n^{2g}\varphi(\mathfrak{n})^k}{\mathfrak{n}}\mathlarger{\prod}_{p\,\mid\, \frac{n}{\mathfrak{n}}}\left(1-\frac{1}{p^{2g}}\right)\mathlarger{\prod}_{p\,\mid\, \mathfrak{n}} \left(1-\frac{(-1)^{k-1}}{(p-1)^{k-1}} \right).
\end{equation}
\end{corollary}

Interestingly, using Theorem~\ref{sepi result}, we can obtain an `equivalent' form of Harvey's theorem (Theorem~\ref{Harvey Thm}). See also \cite{LIS}. Note that conditions (i), (iii) in Corollary~\ref{Harvey Equ} are exactly the same as, respectively, conditions (ii), (iv) in Harvey's theorem.

\begin{corollary} \label{Harvey Equ} Let $\Gamma$ be a co-compact Fuchsian group with signature $(g;n_1,\ldots,n_k)$, and let $\mathfrak{n}:=\textnormal{lcm}(n_1,\ldots ,n_k)$. There is a surface-kernel epimorphism from $\Gamma$ to $\Z_n$ if and only if the following conditions are satisfied:

(i) $\mathfrak{n}\mid n$, and if $g = 0$ then $\mathfrak{n}=n$;

(ii) $e_p>1$ for every prime divisor $p$ of $\mathfrak{n}$;

(iii) if $\mathfrak{n}$ is even then $e_2$ is also even.
\end{corollary}

\begin{proof} The proof simply follows by using the first part of Theorem~\ref{sepi result} and examining the conditions under which the factors of the products in \eqref{sepi_formula} do not vanish.
\end{proof}

It is an interesting problem to develop these counting arguments for the classes of non-cyclic groups. Such results would be very important from several aspects, for example, may lead to more extensions of Harvey's theorem and new proofs for the existing ones, and also may provide us new ways for dealing with the minimum genus and maximum order problems for these classes of groups. So, we pose the following question.

\bigskip

\noindent{\textbf{Problem 1.}} Give explicit formulas for the number of surface-kernel epimorphisms from a co-compact Fuchsian group to a non-cyclic group, say, Abelian, solvable, dihedral, etc.

\section*{Acknowledgements}

The authors would like to thank the anonymous referees for helpful comments. During the preparation of this work the first author was supported by a Fellowship from the University of Victoria (UVic Fellowship).

\end{document}